\documentclass[reqno]{amsart}
\usepackage{amsmath,amssymb}
\usepackage[mathscr]{euscript}
\usepackage{graphicx}
\usepackage{enumitem}
\usepackage[english]{babel}
\usepackage{hyperref}
\usepackage[title]{appendix}

\usepackage[small]{caption}

\usepackage{tikz}

\makeatletter
\@addtoreset{equation}{section}
\makeatother

\newtheorem{theorem}{Theorem}[section]
\newtheorem{lemma}[theorem]{Lemma}

\newtheorem{corollary}[theorem]{Corollary}

\newtheorem{remark}[theorem]{Remark}

\numberwithin{equation}{section}

\newcommand{\bb}[1]{{\mathbb #1}}

\newcounter{as}[section]

\renewcommand{\>}{\rangle}

\begin{document}

\title{Hydrodynamic limit of exclusion processes with slow boundaries on hypercubes}

\author{Tiecheng Xu}

\address{\noindent IME-USP, Rua do Mat\~ao 1010, CEP 05508-090, S\~ao Paulo, Brazil.
  \newline e-mail: \rm \texttt{xutcmath@gmail.com} }

\begin{abstract}
We study the hydrodynamic limit of SSEP with slow boundaries on hypercubes in dimension $d\geq 2$.   The hydrodynamic limit equation is shown to be a heat equation with three different types of boundary conditions according to the slowness of the boundary dynamics. The proof is based on Yau's relative entropy method.
\end{abstract}


\maketitle

\section{Introduction}
The study of scaling limit of the interacting particle systems has achieved a lot of progress in the last decades. In recent years, particle systems with slow dynamics draw a lot of attentions. This type of models is interesting because it presents different behaviors in the macroscopic level depending on how slow the dynamics are put. Fruitful results has been obtained on the hydrodynamic limit(see, for instance \cite{fgn13a,fgs16,ft,bmns,fgn15,egn20a,egn20b,bpgn20,fmn20}), fluctuations of density field(see \cite{fgn13b,fgn19,efgnt20}) and large deviations from the hydrodynamic limit(see \cite{fn17,fgn21}).

Perhapes the first model to study how the strongness of the slow dynamics affect the density evolution in the macroscopic level, is the simple symmetric exclusion process(SSEP) with slow boundaries considered by R. Baldasso et.al in \cite{bmns}. The model is defined as follows. It is a sequence of Markov process indexed by $n$ defined on the set $\{1,2,\cdots, n-1\}$. The dynamics on the bulk is just the well known SSEP. Fix some parameters $\alpha,\beta\in(0,1), c>0$ and $\theta\geq 0$. For the boundary part, particles can enter the system at site $1$ with rate $c\alpha n^{-\theta}$ or leave with rate $c(1-\alpha) n^{-\theta}$ , while at the site $n-1$, particles exhibit a similar behavior, but with rates $c\beta n^{-\theta}$ for entrance and rate $c\beta n^{-\theta}$ for leaving the system. The hydrodynamic limit of the this model is shown to be a heat equation on interval $[0,1]$ with three types of boundary conditions depending on the value of $\theta$ in the diffusive time scale. If $\theta\in [0,1]$, the boundary conditions are is Dirichlet type; if $\theta=1$, the boundary conditions are Robin type; and if $\theta>1$, the boundary conditions are Neumann type . 

Recently, C.Erignoux, P. Gon\c calves and G.Nahum extend the result of \cite{bmns} to the one-dimensional SSEP with rather general slow boundary dymamics\cite{egn20a}\cite{egn20b}. The purpose of this article is to extend the result of \cite{bmns} in another direction, to the high dimension case. The particle systems that we study in this paper evolve on the discrete hypercubes with $(n-1)^d$ points in dimension $d\geq 2$. Similar to the one-dimensional model, the bulk dynamics is just SSEP. For the boundary dynamics, we consider a function $g$ defined on the boundary of the $d$-dimensional continuous hypercube, which plays the same role as $\alpha$ and $\beta$ in the one-dimensional model.  Particles can enter the system at site $x$ on the boundary with rate $cg([x/n]) n^{-\theta}$ if the site is empty, and leave with rate $c(1-g([x/n])) n^{-\theta}$, where $[x/n]$ represents the the point at the boundary of the continuous hypercube with minimum distance to $x/n$. We prove that the hydrodynamic limit of this high dimensional model is still a heat equation with three types of different types of boundary conditions according to the value of $\theta$, which are given in \eqref{<1}, \eqref{=1} and \eqref{>1}.

The main difficulty of our problem is that the invariant measure of boundary driven SSEP cannot be written down explicitly. For the one-dimensional model one could use a proper product measure to approximate the invariant measure with a small enough error in the sense of relative entropy, then apply the Varadhan's entropy method to derive the hydrodynamic limit equation. However it is not longer possible to find a product that can well approximate the invariant measure in the high dimensional case, thus Varadhan's entropy method does not work. What we use in the proof is the celebrated Yau's relative entropy method. One of main novelties of this paper is that time-dependent reference measure is a product measure not only depending on the hydrodynamic equation, but also on the slow parameter $\theta$ if $\theta\in[0,1)$. We are able to show that the relative entropy of state of the process at time $t$  with respect to the time-dependent reference measure is of order $o(n^d)$, which is just enough to derive the hydrodynamic limit. 

The paper is organized as follows. In Section \ref{model} we introduce the model and the result on the relative entropy estimate, then prove the result of hydrodynamic limit. In Section \ref{proof rel} we give the proof of Lemma \ref{lemma} which gives the bound of the relative entropy. We make some long computations used in the proof of Lemma \ref{lemma} in Section \ref{com}. In Section \ref{sp} we prove some properties of the classical solution of hydrodynamic limit equations.

\section{The Model and Main results}\label{model}
\subsection{The model}
We study the simple exclusion processes with slow boundary dynamics on a hypercube in dimension $d\geq 2$. The $d$-dimensional (open) hypercube, denoted by $U$, is the following set
$$U\,=\,\Big\{u\,=\,(u_1,u_2,\dots,u_d)\in \bb R^d\,:\,0< u_i<1\,,\, 1\leq i\leq d\Big\}.$$
Let $\Gamma$ be the boundary of $U$.
The interacting particle systems are defined on the largest lattice contained in $nU$, $n\in\bb N$, denoted by $U_n$, which is defined as follows, 
$$U_n\,:=\,\Big\{x\,=\,(x_1,x_2,\dots,x_d)\in \bb Z^d\,:0< x_i< n\,,\, \forall \, 1\leq i\leq d\Big\}.$$
Let us denote by $\Gamma_n$ the boundary of $U_n$:
$$\Gamma_n\,:=\,\Big\{(x_1,x_2,\dots,x_d)\in U_n\,:\, \exists\, 1\leq i\leq d, \,\text{s.t.}\, x_i=1\,\, \text{or}\,\, x_i=n-1 \Big\}.$$

 The dynamics we consider in this work is the superposition of two dynamics, the bulk dynamics and the boundary dynamics. The bulk dynamics is just the symmetric, simple exclusion process on $U_n$ with reflecting boundary conditions. It is the Markov process on $\Omega_n=\{0,1\}^{U_n}$ whose generator, denoted by $L_{b,n}$, is given by
\begin{equation}
(L_{b,n}f)(\eta)\;=\; \sum_{\substack{|x-y|=1\\x,y\in U_n}}\{f(\sigma^{x,y}\eta)-f(\eta)\}\;.
\end{equation}
In this formula and below, the configurations of $\Omega_n$ are represented by the Greek letters $\eta$, $\xi$, so that $\eta(x)=1$ if site $x\in U_n$ is occupied for the configuration $\eta$ and $\eta(x)=0$ otherwise. The symbol $|\cdot|$ represents the sum norm which is the extension of absolute value to high dimension (so we adopt the same symbol):
$$ |z|\,=\,\sum_{i=1}^d |z_i|.$$
 The symbol $\sigma^{x,y}\eta$ represents the configuration obtained from $\eta$ by exchanging the occupation
variables $\eta(x)$, $\eta(y)$:
\begin{equation*}
(\sigma^{x,y}\eta)(z)=
\begin{cases}
\eta(y) & \mbox{ if } z=x\\
\eta(x) & \mbox{ if } z=y\\
\eta(z) & \mbox{ if } z\in U_n\setminus \{x,y\}\;. 
\end{cases}
\end{equation*}

The boundary dynamics on $\Gamma_n$ are non-conservative.
We start by fixing $\theta\geq 0$ which we call the \textbf{slow parameter}, $c>0$ and a function $g:  \Gamma \to \bb (0,1)$. For each $u\in U$, let $[u]$ be the point at the boundary $\Gamma$ with minimum distance to $u$. If there is more than one such points, we choose any one of them. For each site $z\in \Gamma_n$, if it is occupied, then with rate $cn^{-\theta}(1-g([z/n]))$ the particle of this site is eliminated; if it is empty, then with rate $cn^{-\theta}g([z/n])$, one particle is created at this site. This dynamics is a Markov process with generator $L_{p,n}$ given by
$$(L_{p,n})f(\eta)=cn^{-\theta}\sum_{z\in\Gamma_n}[g([z/n])(1-\eta(z))+(1-g([z/n]))\eta(z)][f(\sigma^z\eta)-f(\eta)]$$
where $\sigma^z\eta$, $z\in \Gamma_n$, is the configuration obtained
from $\eta$ by flipping the occupation variable $\eta(z)$,
\begin{equation*}
(\sigma^z\eta)(x)=\begin{cases}
1-\eta(z) & \mbox{ if } x=z\\
\eta(x) & \mbox{ if } x\in \Gamma_n\setminus \{z\}\;. 
\end{cases}
\end{equation*}

For each $n\in\bb N$, denote by $\{\eta_t^n : t\ge 0\}$ the  Markov process with generator $n^2 L_n=n^2L_{b,n}+n^2L_{p,n}$. Here we accelerate the process by $n^2$ because the evolution of the density in the macroscopic level is observed on the diffusive timescale. Note that the invariant measure of the Markov process $\{\eta_t^n : t\ge 0\}$ can not be written down explicitly in general, except for some special choice of $g$. This causes the main difficulty for deriving the hydrodynamic limit.

\subsection{A  key mapping and the reference measure}
Assume that the process $\{\eta_t^n: t\geq 0 \}$ starts from an initial distribution $\mu^n$. Fix $T>0$. Denote by $\mu^n_t$ the distribution of the speed-up process at time $t$, for all $t\in(0,T]$. Given a function $\rho(\cdot):U\to \bb [0,1]$,  we denote by $\nu_\rho^n$ the Bernoulli product measure with slowly varying parameter associated to the profile $\rho(\cdot)$:
\begin{equation}\label{defnu}
\nu_{\rho(\cdot)}^n\{\eta:\eta(x)=1\}\,=\,\rho(x/n), \,\,\,\,x\in  U_n.
\end{equation}

We will apply the celebrated Yau's relative entropy method to derive the hydrodynamic limit. One of the main steps is to choose a proper time-dependent reference measure $\nu_t^n$ which is close to $\mu_t^n$,  in the sense of relative entropy. A popular choice is $\nu_t^n=\nu^n_{\rho_t(\cdot)}$, where $\rho_t(\cdot)$ is the solution of the expected hydrodynamic equation of the particles system. However this choice turns out to be not good enough to reach the desired bound of the relative entropy in our model. A modification to that choice is needed according to the slow parameter $\theta$. To explain our choice of the reference measure $\nu_t^n$, we introduce the following mapping. 

We need to define some sets first. Define
$$\Gamma_n^{i,+}\,:=\,\Big\{(x_1,x_2,\cdots,x_d)\in\Gamma_n\,:\, x_i=1\Big\},$$
$$\Gamma_n^{i,-}\,:=\,\Big\{(x_1,x_2,\cdots,x_d)\in\Gamma_n\,:\, x_i=n-1\Big\},$$
and
$$\Gamma_n^i\,:=\,\Gamma_n^{i,+}\cup\Gamma_n^{i,-}.$$
It is obvious that $\Gamma_n=\cup_{i=1}^d \Gamma_n^i$. We also need to define the following set which can be thought as the closure in $\bb Z^d$ of $U_n$:
$$\overline{U}_n\,:=\,\Big\{x\,=\,(x_1,x_2,\dots,x_d)\in \bb Z^d\,:0\leq x_i\leq n\,,\,\forall\,1\leq i\leq d\Big\}.$$ 
Recall that $\Gamma$ is the boundary of set $U$.  Let $\overline{U}$ be the closure of $U$:
$\overline{U}:=U\cup \Gamma$.

There is a natural partial order $\prec$ on $\bb R^d$: given $u^1,u^2\in \bb R^d$,  $u^1\prec u^2$ if and only if, $u^1_i\leq u^2_i$ for every $1\leq i\leq d$ and the strict inequality achieves for at least one coordinate $i$. Now we can describe the key mapping that will be used to define our reference measure $\nu_t^n$. For each $n\in\bb N$, it is a mapping from $\overline{U}_n$ to $\overline{U}$. Let us denote it by $m_n$. Firstly, $m_n$ preserves the partial order $\prec$: if $x,y\in \overline{U}_n$ such that $x\prec y$, then $m_n(x)\prec m_n(y)$. Moreover, $m_n$ satisfies the following additional conditions. Let $\{e_i: 1\leq i\leq d\}$ be the collection of canonical basis of $\bb R^d$.
If $\theta\in[0,1)$, then\,:
\begin{enumerate}
\item  For every $x\in \overline{U}_n\backslash U_n$,  \,$m_n(x)= \frac{x}{n}$.
\item  For every $x\in\Gamma_n^i$, $1\leq i\leq d$, if $x_i=n-1$, then 
$$m_n(x+e_i)\,-\,m_n(x)\,=\,c^{-1}\, n^{\theta-1};$$  
 if $x_i=1$, then 
$$m_n(x)\,-\,m_n(x-e_i)\,=\,c^{-1} \,n^{\theta-1}.$$ 
\item  For every $x\in U_n\backslash\Gamma_n$, for every $1\leq i\leq d$,
$$m_n(x)\,-\,m_n(x-e_i)\,=\,m_n(x+e_i)\,-\,m_n(x)\,=\,\frac{1-2c^{-1}n^{\theta-1}}{n-2}\,.$$
\end{enumerate}
If $\theta\geq1$, we just let $m_n(x):=\frac{x}{n}$ for all $x\in \overline{U}_n$. It is not hard to see that for each $\theta\geq 0$, $m_n$ is uniquely determined by the above conditions. Actually it is also possible to compute the explicit value of $m_n(x)$. Since only the relations above are needed in our proof, we are not going to compute the explicit expression of $g_n$.

Now let us define the time-dependent reference measure $\nu_t^n$, $0\leq t\leq T$. It is a Bernoulli product measure with density $\rho(t,m_n(x))$ at site $x\in U_n$:
\begin{equation}\label{nu}
\nu_t^n(\eta)=\prod_{x\in U_n}\Big\{\eta(x)\rho(t,m_n(x))\,+\,(1-\eta(x))(1-\rho(t,m_n(x)))\Big\}.
\end{equation}
 In the last formula, $\rho(t,\cdot)$ is the hydrodynamic limit equation that will be given in the next subsection. Note that for the case $\theta\geq1$, $\nu_t^n$ is exactly the popular choice $\nu^n_{\rho_t(\cdot)}$, by definition of $m_n$. However, for the case $\theta\in[0,1]$, the distance $m_n(x)$ with $x\in\Gamma_n$ to $\Gamma$ is no longer $1/n$, but is adjusted accordingly. This plays a crucial role in the proof of Lemma \ref{order}.

\subsection{Hydrodynamic limit equations}\label{hle}

To describe the hydrodynamic limit equation, the boundary $\Gamma$ needs to be treated carefully piece by piece. For each $1\leq i\leq d$, let $\Gamma^{i,+}$ ($\Gamma^{i,-}$, respectively) be the hyperplane on which the $i-$coordinate of sites are equal to $0$ ($1$, respectively):
$$\Gamma^{i,+}\,:=\,\Big\{(r_1,r_2,\cdots,r_d)\in\Gamma\,:\, r_i=0\Big\}.$$
Let $\Gamma^i$ be the union of these two disjoint sets $\Gamma^{i,+}$ and $\Gamma^{i,-}$. It is clear that 
$$\Gamma\,=\,\cup_{1\leq i\leq d} \,\Gamma^i.$$
Note that for different $i$ and $j$, $\Gamma_i$ and $\Gamma_j$ are not necessarily disjoint.

Like the one-dimensional case considered in \cite{bmns}, the hydrodynamic limit equation of our model depends on the slow parameter $\theta$: 
\begin{itemize}

\item if $\theta\in[0,1),$
\begin{equation}
\label{<1}
\begin{cases}
\partial_t\rho(t,u)=\Delta \rho(t,u), \quad u\in U \\
\rho(t,u)=g(u), \quad u\in\Gamma \\
\rho(0,\cdot)=\rho_0(\cdot), \quad u\in U \,
\end{cases}
\end{equation}

\item if $\theta=1$,
\begin{equation}
\label{=1}
\begin{cases}
\partial_t\rho(t,u)=\Delta \rho(t,u) \quad u\in U  \\
\rho(0,\cdot)=\rho_0(\cdot),\quad u\in U \,\\
\partial^-_{u_i}\rho(t,u)=c\,\big(g(u)-\rho(t,u)\big), \quad u\in\Gamma^{i,-}, \quad 1\leq i\leq d\\
\partial^+_{u_i}\rho(t,u)=c\,\big(\rho(t,u)-g(u)\big), \quad u\in\Gamma^{i,+},\quad 1\leq i\leq d
\end{cases}
\end{equation}

\item if $\theta>1$,
\begin{equation}
\label{>1}
\begin{cases}
\partial_t\rho(t,u)=\Delta \rho (t,u),\quad u\in U \,, \\
\rho(0,\cdot)=\rho_0(\cdot),\quad u\in U \,\\
\partial^+_{u_i}\rho(t,u)=\,0, \quad u\in\Gamma^{i,+}, \quad 1\leq i\leq d\\
\partial^-_{u_i}\rho(t,u)=\,0, \quad u\in\Gamma^{i,-},\quad 1\leq i\leq d

\end{cases}
\end{equation}

\end{itemize}

In order to gaurantee the solution $\rho(t,\cdot)$ has good properties, we assume that 

\begin{enumerate}
\item The initial profile $\rho_0$ is smooth on the closed set $\overline{U}$:
\begin{equation*}
\rho_0\in C^\infty(\overline{U}).
\end{equation*}
Moreover, there exists an open set $V\subset\bb R^d$ covering $\Gamma$ and an extension $\widetilde{g}$ of $g$ such that $\widetilde{g}$ is smooth on $V$.

\item 
There exists $\varepsilon_0\in (0,\frac{1}{2})$ such that both $\rho_0$ and $g$ take value in $(\varepsilon_0,1-\varepsilon_0)$. 

\item
$g$ and $\rho_0$ satisfy the compatibility condition of the PDE at $t=0$. Namely, if $\theta\in[0,1)$, $\rho_0(u)=g(u)$ for all $u\in\Gamma$. if $\theta=0$, $\partial^\pm_{u_i}\rho_0(u)=\pm c\big(\rho(t,u)-g(u)\big)$ for every $u\in\Gamma^{i,\pm}$ and every $1\leq i\leq d$. If $\theta>1$, $\partial^\pm_{u_i}\rho_0(u)=0$ for every $u\in\Gamma^{i,\pm}$ and every $1\leq i\leq d$.

\end{enumerate}

In this work, we need to deal with the differentiability of functions on some hypercube $V$, which could be open, closed or neither open nor closed. Suppose $V\in\bb R^d$. Given a function $f:V\to\bb R$, define the partial derivative $\partial_{u_i}f$ at $u\in V$ by:
$$\partial_{u_i}f(u)\,=\,\lim_{\substack{h\to0\\u+he_i\in V}}\frac{f(u+he_i)\,-\,h(u)}{h}.$$
Note that if $u$ is not contained in the interior of $V$, for some $i$, $h$ can only choose to be either positive or negative in order to guarantee $u+he_i\in V$. With this definition, a function $f$ is said to be $C^1$ continuous on $V$ if for every $1\leq i\leq d$, $\partial_{u_i}f$ exists everywhere and is continuous on $V$. In the same manner, we could define $C^k$ continuous functions on $V$ for integers $k\geq 1$ inductively and then the smooth functions.

\subsection{Main results}
 Denote by $\pi_t^n$ the empirical measure on $U$ associated to the configuration $\eta_t$:
$$\pi_t^n(du)= \pi^n(\eta_t,du):=\frac{1}{|U_n|}\sum_{x \in U_n}\eta_t(x)\delta_{x/N}(du)$$
where $\delta_x$ represents the Dirac mass at $x$. Throughout this article, for a set $A\subset \bb Z^d$, $|A|$ always represents the total number of sites contained in $A$. Given a continuous function $H:U \to \mathbb{R}$, denote by $\langle \pi_t^n,H \rangle$ the integral of $H$ with respect to $\pi_t^n$:
$$\langle \pi_t^n,H \rangle\,:=\,\frac{1}{|U_n|}\sum_{x \in U_n}H\big(\frac{x}{n}\big)\eta_t(x) $$.

Given two probability measures $\mu$ and $\pi$ on the same probability space $E$ such that, $\mu$ is absolutely continuous with respect to $\pi$, the relative entropy of $\mu$ with respect to $\pi$, denoted by $H(\mu|\pi)$, is defined by 
$$H(\mu|\pi)\,=\,\int_E\frac{d\mu}{d\pi}\log\frac{d\mu}{d\pi}d\pi,$$
where $\frac{d\mu}{d\pi}$ represents the Radon-Nikodym derivative of $\mu$ with respect to $\pi$.

 Now we are ready to state our first main result.
\begin{lemma}\label{lemma}
Assume that $\rho_0$ and $g$ satisfy assumptions (1)(2)(3) made in subsection \ref{hle} and recall the definition of $\nu_t^n$ given in \ref{nu}. Assume also that the relative entropy of the initial distribution $\mu_n$ with respect to the measure $\nu_0^n$ is of order $o(n^d)$:
\begin{equation}\label{initial}
H(\mu^n|\nu_0^n)\,=\,o(n^d).
\end{equation}
Then the relative entropy of the state of the process at time $t$ with respect to $\nu_t^n$ is also of order $o(n^d)$:
\begin{equation}\label{end}
H(\mu^n_t|\nu_t^n)\,=\,o(n^d).
\end{equation}

\end{lemma}
The proof of this lemma is given in Section \ref{proof rel}. We do not claim the bound in \eqref{end} is sharp when the initial bound in \eqref{initial} is small enough. It should be possible to apply the method introduced in the remarkable work of M. Jara and O.Menezes \cite{jm} to improve the bound. Nevertheless, the estimate of the relative entropy in the lemma is already good enough to derive the hydrodynamic limit of our model.

\begin{theorem}\label{HDL}
Under assumptions of the previous lemma, for every continuous function $G: U\to\bb R$,
\begin{equation}
\lim_{n\to\infty} E_{\mu_t^n}\Big[\Big|\langle \pi_t^n,G\rangle\,-\,\int_U G(u)\rho(t,u)du\Big|\Big]\,=\,0
\end{equation}
where $\rho(t,\cdot)$ is the PDE given in the previous subsection \ref{hle}.

\end{theorem}

Before proving this theorem, let us recall the concept of subgaussian variable which will be used in the proof. We say that a real-valued random variable $X$ is subgaussian of order $\sigma^2$, if for every $\theta\in \bb R$,
$$\log E[e^{\theta X}]\,\leq\,\frac{1}{2}\sigma^2\theta^2.$$

\begin{proof}[Proof of Theorem \ref{HDL}]
Note that by definition of $m_n$, for every $x,y\in U_n$ such that $|x-y|=1$, we have $|m_n(x)-m_n(y)|=o(1)$. Since both $G$ and $\rho(t,\cdot)$ are continuous, 
$$\lim_{n\to\infty}\frac{1}{|U_n|}\sum_{x\in U_n}G\big(\frac{x}{n}\big)\rho(t,m_n(x))\,=\,\int_U G(u)\rho(t,u)du.$$
Therefore to prove the theorem, it is enough to show that 
\begin{equation}
\lim_{n\to\infty} E_{\mu_t^n}\Big[\Big|\frac{1}{|U_n|}\sum_{x\in U_n}G\big(\frac{x}{n}\big)\big[\eta_t(x)\,-\,\rho(t,m_n(x))\big]\Big|\Big]\,=\,0.
\end{equation}
Applying the entropy inequality, for every $\gamma>0$, the expectation in the previous formula is bounded from above by 
\begin{equation}\label{entropy inequality}
\frac{1}{\gamma |U_n|}\left\{H(\mu_t^n|\nu_t^n)\,+\,\log E_{\nu_t^n}\Big[exp\Big|\big\{\gamma \sum_{x\in U_n}G\big(\frac{x}{n}\big)\big[\eta_t(x)\,-\,\rho(t,m_n(x))\big]\Big|\big\}\Big]\right\}.
\end{equation}
Since $e^{|x|}\leq e^x+e^{-x}$ and 
$$\limsup_{n}\frac{1}{n}\log(a_n+b_n)\,=\,\max\big\{\limsup_n \frac{\log a_n}{n}\,,\,\limsup_n\frac{\log b_n}{n}\big\},$$
 we can remove the absolute value symbol inside the exponential.

By Hoeffding's lemma(see Lemma F.9 in \cite{jm}, for instance), since $\eta$ takes value in $[0,1]$ and $\rho\in(\varepsilon_0,1-\varepsilon_0)$ by Corollary \ref{col}, under the product measure  $\nu_t^n$, $\eta_t(x)-\rho(t,m_n(x))$ is subgaussian of order $\frac{1}{4}$. By Lemma F.12 in \cite{jm}, $\sum_{x\in U_n}G(\frac{x}{n})\big[\eta_t(x)\,-\,\rho(t,m_n(x))\big]$ is subgaussian of order $\sum_{x\in U_n}\frac{G(x/n)^2}{4}$. Therefore, by definition of the subgaussian variable, the expression in \eqref{entropy inequality} is bounded by
$$\frac{1}{\gamma |U_n|}H(\mu_t^n|\nu^n_t)\,+\,\frac{\gamma}{ |U_n|}\sum_{x\in U_n}\frac{G(x/n)^2}{8}.$$
Choosing $\gamma=\sqrt{\frac{H(\mu_t^n|\nu^n_t)}{|U_n|}}$, we have that
\begin{equation}\label{repbound}
E_{\mu_t^n}\Big[\Big|\frac{1}{|U_n|}\sum_{x\in U_n}G(\frac{x}{n})\big[\eta_t(x)\,-\,\rho(t,m_n(x))\big]\Big|\Big]\,\leq\, \big(\|G\|_\infty+1\big)\sqrt{\frac{H(\mu_t^n|\nu^n_t)}{|U_n|}}
\end{equation}
To conclude the proof of this theorem, it remains to recall \eqref{end} .
\end{proof}
\begin{remark}
The one-dimensional case of this problem was solved by R.Baldasso et.al in \cite{bmns}, by applying the Varadhan's entropy method. However, that method does not work here. The main reason is that for any product measure $\nu^n$ and density function $f$, since $\nu^n$ is not invariant, $n^2\<L_nf,f\>_{\nu^n}$ can only be bounded by a positive term of order $O(n^d)$, which is not enough to prove the local replacement lemma if $d\geq 2$. Without the local replacement lemma, one cannot characterize the boundary behaviour of the dynamics in the limit as $n\to\infty$. Nevertheless, when the invariant measure is product, the Varadhan's entropy method still works for deriving the hydrodynamic limit of high dimensional models. This is the case of the model in \cite{ft}, where the authors study the SSEP with a slow membrane in dimension $d\geq 2$. 
\end{remark}

\section{Relative entropy estimate}\label{proof rel}

In this section we give the proof of Lemma \ref{lemma}. The idea is  to show that the time derivative of the relative entropy of $\mu_t^n$ with respect to $\nu^n_t$, is bounded by a term of order $o(n^d)$ plus $H(\mu^n_t|\nu_t^n)$. Once this is proved, Lemma \ref{lemma} follows immediately from the Gronwall lemma and  assumption \eqref{initial}.

Fix a reference measure $\nu^n$ on the state space $\Omega_n$. Recall our choice of $\nu_t^n$ in \eqref{nu}. Let $f_t^n$ be the Radon-Nikodym derivative of $\mu_t^n$ with respect to $\nu^n$:
$$f_t^n\,:=\,\frac{d\mu_t^n}{d\nu^n},$$
and let $\Psi_t^n$ be the Radon-Nikodym derivative of $\nu^n_t$ with respect to $\nu^n$
 $$\Psi_t^n\,:=\,\frac{d\nu^n_t}{d\nu^n}.$$
The lemma below gives an upper bound for the time derivation of the relative entropy $ H'(\mu^n_t|\nu_t^n)$ in terms of $f_t^n$ and $\Psi_t^n$.
\begin{lemma}\label{deriv}
For every $t>0$,
$$ H'(\mu^n_t|\nu_t^n)\,\leq\, \int \frac{1}{\Psi_t^n}\{n^2L^*_n\Psi_t^n-\partial_t \Psi_t^n\} f_t^n d\nu^n,$$
where $L^\star_n$ is the adjoint generator of $L_n$ with respect to the measure $\nu^n$.
\end{lemma}
This is a classical result, which can be found in Chapter 6 of the book \cite{kl}, as Lemma 1.4.
It is worthwhile to point out that the right hand side of the inequality in the previous lemma does not depend on the choice of the reference measure $\nu^n$. However a good choice of $\nu^n$ would simplify the computation. Here we choose $\nu^n$ to be the Bernoulli product measure with slowly varying parameter associated to the constant profile $1/2$:
$$\nu^n(\eta)=\prod_{x\in U_n}\Big\{\eta(x)\frac{1}{2}\,+\,(1-\eta(x))\frac{1}{2}\Big\}.$$
In other words, $\nu^n$ is the uniform measure on $\Omega_n$.

Now we just need to bound the right hand side of the inequality in Lemma \ref{deriv} by the sum of a term of order $o(n^d)$ and $H(\mu^n_t|\nu_t^n)$.
\begin{lemma}\label{order}
For the choice of $\nu_t^n$ given in \eqref{nu}, there exists a constant $C$ independent of $n$, such that for every $t\in [0,T]$, 
$$\int \frac{1}{\Psi_t^n}\{n^2L^*_n\Psi_t^n-\partial_t \Psi_t^n\} f_t^n d\nu^n\,\leq\, o(n^d)\,+\,H(\mu_t^n|\nu^n_t).$$
where $L^*_n$ is the adjoint generator of $L_n$ w.r.t. $\nu^n.$
\end{lemma}
The proof of this lemma is postponed to the end of Section \ref{com}.

\section{Some computations}\label{com}
\subsection{Direct computation of $\frac{n^2L^*_n\Psi_t^n}{\Psi_t^n}$}
Since $\nu_t^n$ is a product measure, a simple computation gives that
$$\Psi_t^n(\eta)=\prod_{x\in U_n}\eta(x)\frac{\rho(t,m_n(x))}{1/2}+(1-\eta(x))\frac{1-\rho(t,m_n(x))}{1/2}.$$
Applying formula A.1 in \cite{jm}, for every function $h:\Omega_n\to\bb R$,
$$L_{b,n}^\star h(\eta)\,=\,\sum_{\substack{|x-y|=1\\x,y\in U_n}}[h(\eta^{x,y})-h(\eta)]$$
and
\begin{equation*}
\begin{split}
L_{p,n}^\star h(\eta)\,=\,cn^{-\theta}\sum_{x\in\Gamma_n}&\Big[\big\{g([x/n])\eta(x)+(1-g([x/n]))(1-\eta(x))\big\}\frac{d\nu^n(\sigma^x\eta)}{d\nu^n(\eta)}h(\sigma^x\eta)\,\\
&-\,\big\{g(x/n)(1-\eta(x))+(1-g(x/n))\eta(x)\big\}h(\eta)\Big],
\end{split}
\end{equation*}
where $L_{b,n}^\star$(resp. $L_{p,n}^\star$) is the adjoint operator of $L_{b,n}$(resp. $L_{p,n}$) with respect to the probability measure $\nu^n$.
For any configuration $\eta\in\Omega_n$ and $x\in U_n$, let us denote
$$P_x(\eta)=g([x/n])(1-\eta(x))+(1-g([x/n]))\eta(x).$$ 
With this notation and recalling that $L_n=L_{b,n}+L_{p,n}$,
\begin{equation}\label{Psi1}
\begin{split}
\frac{1}{\Psi_t^n}n^2L^*_n\Psi_t^n(\eta)\,=&\,n^2\sum_{\substack{|x-y|=1\\x,y\in U_n}}[\frac{\Psi_t^n(\eta^{x,y})}{\Psi_t^n(\eta)}-1]\\
+&\,cn^{2-\theta} \sum_{x\in \Gamma_n}[P_x(\sigma^x\eta)\frac{d\nu^n(\sigma^x\eta)}{d\nu^n(\eta)}\frac{\Psi_t^n(\sigma^x\eta)}{\Psi_t^n(\eta)}-P_x(\eta)].
\end{split}
\end{equation}

A simple computation shows that if $x,y\in U_n$ and $|x-y|=1$, then
\begin{equation*}
\begin{split}
\frac{\Psi_t^n(\eta^{x,y})}{\Psi_t^n(\eta)}\,&=\,\eta(x)\eta(y)+(1-\eta(x))(1-\eta(y))\\
&+\,\eta(x)(1-\eta(y))\frac{\big[1-\rho(t,m_n(x))\big]\rho(t,m_n(y))}{\rho(t,m_n(x))\big[1-\rho(t,m_n(y))\big]}\\
&+\,\eta(y)(1-\eta(x))\frac{\big[1-\rho(t,m_n(y))\big]\rho(t,m_n(x))}{\rho(t,m_n(y))\big[1-\rho(t,m_n(x))\big]}.
\end{split}
\end{equation*}
On the other hand, if $x\in \Gamma_n$, then 
$$\frac{\Psi_t^n(\sigma^x\eta)}{\Psi_t^n(\eta)}\,=\,\eta(x)\frac{1-\rho(t,m_n(x))}{\rho(t,m_n(x))}+(1-\eta(x))\frac{\rho(t,m_n(x))}{1-\rho(t,m_n(x))}.$$
Putting these two identities above into \eqref{Psi1}, we have that 
\begin{equation}\label{4sum}
\frac{1}{\Psi_t^n}n^2L^*_n\Psi_t^n(\eta)\,=\,A_n\,+\,B_n\,+\,C_n\,+\,D_n,
\end{equation}
where
$$A_n\,=\,\frac{n^2}{2}\sum_{x\in U_n}\sum_{\substack{|z|=1\\x+z\in U_n}}(1-\eta(x))\eta(x+z)\Big[\frac{\big[1-\rho(t,m_n(x+z))\big]\rho(t,m_n(x))}{\rho(t,m_n(x+z))\big[1-\rho(t,m_n(x))\big]}\,-\,1\Big],$$
$$B_n\,=\,\frac{n^2}{2}\sum_{x\in U_n}\sum_{\substack{|z|=1\\x+z\in U_n}}\eta(x)(1-\eta(x+z))\Big[\frac{\big[1-\rho(t,m_n(x))\big]\rho(t,m_n(x+z))}{\rho(t,m_n(x))\big[1-\rho(t,m_n(x+z))\big]}\,-\,1\Big],$$
$$C_n\,=\,cn^{2-\theta}\sum_{x\in\Gamma_n}\eta(x)\left[P_x(\sigma^x\eta)\frac{1-\rho(t,m_n(x))}{\rho(t,m_n(x))}-P_x(\eta)\right],$$
$$D_n\,=\,cn^{2-\theta}\sum_{x\in\Gamma_n}(1-\eta(x))\left[P_x(\sigma^x\eta)\frac{\rho(t,m_n(x))}{1-\rho(t,m_n(x))}-P_x(\eta)\right].$$
There is a factor $\frac{1}{2}$ in the expressions of $A_n$ and $B_n$ because each edge $(x,y)\subset U_n$ is counted twice if we count according to endpoints.

\subsection{Estimate of $\frac{n^2L^*_n\Psi_t^n}{\Psi_t^n}$}

For each $t\in[0,T]$, define the function $H(t,\cdot):\overline{U}\to\bb R$ by 
$$H(t,u):=\log \frac{\rho(t,u)}{1-\rho(t,u)}.$$
Note that $H(t,\cdot)$ is well defined and smooth on $U$ by Corollary \ref{col}.  It is easy to check that for every $u\in U$, for every $1\leq j \leq d$,
\begin{equation}\label{rel1}
\partial_{u_j} H(t,u)=\frac{\partial_{u_j}\rho(t,u)}{\rho(t,u)\big(1-\rho(t,u)\big)}
\end{equation}
and 
\begin{equation}\label{rel2}
\partial ^2_{u_ju_j}H(t,u)=\frac{\partial_{u_ju_j}^2\rho(t,u)}{\rho(t,u)\big(1-\rho(t,u)\big)}-\frac{(\partial_{u_j}\rho(t,u))^2\big(1-2\rho(t,u)\big)}{\rho(t,u)^2\big(1-\rho(t,u)\big)^2}.
\end{equation}
With those notation above, we can rewrite $A_n+B_n$ as 
\begin{equation*}
\begin{split}
&\frac{n^2}{2}\sum_{x\in U_n}\sum_{\substack{|z|=1\\ x+z\in U_n}}\eta(x+z)(1-\eta(x))\Big[\exp\Big\{H(t,m_n(x))-H(t,m_n(x+z))\Big\}-1\Big]\\
+\,&\frac{n^2}{2}\sum_{x\in U_n}\sum_{\substack{|z|=1\\ x+z\in U_n}}\eta(x)(1-\eta(x+z))\Big[\exp\Big\{H(t,m_n(x+z))-H(t,m_n(x))\Big\}-1\Big].
\end{split}
\end{equation*}
The next lemma gives an estimate of $A_n+B_n$ of the one-dimensional version.
\begin{lemma}\label{computation}
Fix a smooth function $\rho:[0,1]\to\bb (0,1)$. Define a function $H:[0,1]\to\bb R$ by $H(\cdot)=\log\frac{\rho(\cdot)}{1-\rho(\cdot)}$. For each $n\in\bb N$, consider a mapping 
$m_n:\{0,1,\cdots, n\}\to[0,1]$
such that $m_n(0)<m_n(1)<\cdots<m_n(n)$, 
\begin{equation}\label{mid}
\lim_{n\to\infty}n\big\{m_n(j+1)-m_n(j)\big\}=1 \quad \text{for every }\,\, 1\leq j\leq n-2.
\end{equation}
Then for integers $n$ sufficiently large, 
\begin{equation}\label{AB1}
\begin{split}
&n^2\sum_{j=1}^{n-2}\eta(j+1)(1-\eta(j))\left[\exp\Big\{H(m_n(j))\,-\,H(m_n(j+1))\Big\}\,-\,1\right]\\
+\,&n^2\sum_{j=1}^{n-2}\eta(j)(1-\eta(j+1))\left[\exp\Big\{H(m_n(j+1))\,-\,H(m_n(j))\Big\}\,-\,1\right]
\end{split}
\end{equation}
is equal to 
\begin{equation*}
\begin{split}
&n\Big\{[\eta(1)-\rho(0)]H'(m_n(1))\,-\,[\eta(n-1)-\rho(1)]H'(m_n(n-1))\Big\}\\
+\,&\sum_{j=1}^{n-2}\Big[\frac{1}{2}\big(\eta(j)+\eta(j+1)-2\eta(j)\eta(j+1)\big)\,-\,\rho\big(\frac{j}{n}\big)\{1-\rho\big(\frac{j}{n}\big)\}\Big]\big(H'(m_n(j))\big)^2\\
+\,&\sum_{j=2}^{n-1}\big[\eta(j)\,-\,\rho\big(\frac{j}{n}\big)\big]H''(m_n(j))\,+\,O(1).
\end{split}
\end{equation*}
\end{lemma}
\begin{proof}
Applying Taylor expansion to the exponential function, since $\rho$ is smooth, by \eqref{mid}, the sum  in \ref{AB1} is equal to
\begin{equation*}
\begin{split}
&n^2\sum_{j=1}^{n-2}\Big[\eta(j)-\eta(j+1)\Big]\Big[H(m_n(j+1))-H(m_n(j))\Big]\,+\,O(1)\\
+\,&\frac{n^2}{2}\sum_{j=1}^{n-2}\Big[\eta(j)+\eta(j+1)-2\eta(j)\eta(j+1)\Big]\Big[H(m_n(j+1))-H(m_n(j))\Big]^2.
\end{split}
\end{equation*}
Applying Taylor expansion to the function $H$, the sum above is equal to 
\begin{equation*}
\begin{split}
&n^2\sum_{j=1}^{n-2}\Big[\eta(j)-\eta(j+1)\Big]\Big(\frac{1}{n}H'(m_n(j))+\frac{1}{2n^2}H''(m_n(j))\Big)\\
+\,&\frac{n^2}{2}\sum_{j=1}^{n-2}\Big[\eta(j)+\eta(j+1)-2\eta(j)\eta(j+1)\Big]\Big(\frac{1}{n}H'(m_n(j))\Big)^2\,+\,O(1).
\end{split}
\end{equation*}
By a summation by parts and Taylor expansion to $H'(\cdot)$, we have 
\begin{equation*}
\begin{split}&n^2\sum_{j=1}^{n-2}\big[\eta(j)-\eta(j+1)\big]\frac{1}{n}H'(m_n(j))\\
=\,&n\Big\{\eta(1)H'(m_n(1))-\eta(n-1)H'(m_n(n-1))+\sum_{j=2}^{n-1}\eta(j)\frac{1}{n}H''(m_n(j))\Big\}+O(1)
\end{split}
\end{equation*}
and 
$$n^2\sum_{j=1}^{n-2}[\eta(j)-\eta(j+1)]\frac{1}{2n^2}H''(m_n(j))=O(1).$$
Up to now, we have proved that the sum in \ref{AB1} is equal to
\begin{equation}\label{sum1}
\begin{split}
&n\Big\{\eta(1)H'(m_n(1))-\eta(n-1)H'(m_n(n-1))+\sum_{j=2}^{n-1}\eta(j)\frac{1}{n}H''(m_n(j))\Big\}\\
+\,&n^2\sum_{j=1}^{n-1}[\eta(j)+\eta(j+1)-2\eta(j)\eta(j+1)]\frac{1}{2}\Big(\frac{1}{n}H'(m_n(j))\Big)^2+O(1)
\end{split}
\end{equation}

Since the function $H$ and $\rho$ are smooth on $[0,1]$, by relation \eqref{rel1}, we obtain a discrete version of the integral by parts formula:
\begin{equation*}
\begin{split}
&\sum_{j=2}^{n-2}H''(m_n(j))\rho\big(\frac{j}{n}\big)+\sum_{j=1}^{n-1}\big(H'(m_n(j))\big)^2\rho\big(\frac{j}{n}\big)\{1-\rho\big(\frac{j}{n}\big)\}\\
=&\,n\big[\rho(\frac{n-1}{n}) H'(m_n(n-1))-\rho(\frac{1}{n}) H'(m_n(1))\big]+O(1).
\end{split}
\end{equation*}
Substracting the above identity in equation \eqref{sum1}, the sum in \ref{AB1} is equal to
\begin{equation*}
\begin{split}
&n\Big\{\big[\eta(1)-\rho\big(\frac{1}{n}\big)\big]H'(m_n(1))\,-\,\big[\eta(n-1)-\rho\big(\frac{n-1}{n}\big)\big]H'(m_n(n-1))\Big\}\\
+\,&\sum_{j=1}^{n-2}\Big[\frac{1}{2}(\eta(j)+\eta(j+1)-2\eta(j)\eta(j+1))\,-\,\rho\big(\frac{j}{n}\big)\{1-\rho\big(\frac{j}{n}\big)\}\Big]\big(H'(m_n(j))\big)^2\\
+\,&\sum_{j=1}^{n-1}\big[\eta(j)\,-\,\rho\big(\frac{j}{n}\big)\big]H''(m_n(j))\,+\,O(1).
\end{split}
\end{equation*}
\end{proof}

For any point $x=(x_1,x_2,\cdots,x_n)\in U_n$ and $1\leq i\leq d$, we define a one-dimensional set which contains $x$\,: 
$$S_x^i\,:=\,\{y\in U_n:\, y_j=x_j, \quad \forall \, j\neq i\}.$$
Observe that for each $S_x^i$, there is a unique point $x'\in \Gamma_n^{i,+}$ such that $S_{x'}^i=S_x^i$. From this observation we see that the total number of set $S_x^i$ is at most  $d n^{d-1}$.
Applying Lemma \ref{computation} to each set $S_x^i$, we have
\begin{equation}\label{AB}
\begin{split}
A_n+B_n&\,=\,n\sum_{i=1}^d\sum_{x\in\Gamma_n^{i,+}}\big[\eta(x)-g([\frac{x}{n}])\big]\partial_{u_i}H(t,m_n(x))\\
&\,-\,n\sum_{i=1}^d\sum_{x\in\Gamma_n^{i,-}}\big[\eta(x)-g([\frac{x}{n}])\big]\partial_{u_i}H(t,m_n(x))\\
&\,+\,\sum_{i=1}^d \sum_{x\in\Gamma_n^{i,+}}\sum_{j=0}^{n-2}R^t_n(x+je_i)\big(\partial_{u_i}H(t,m_n(x+je_i))\big)^2\\
&\,+\,\sum_{x\in U_n}\big[\eta(x)\,-\,\rho(t,m_n(x))\big]\sum_{i=1}^d\partial_{u_i u_i}^2H(t,m_n(x))\,+\, o(n^d)\\
\end{split}
\end{equation}
where for a configuration $\eta\in\Omega_n$ and a fixed $1\leq i\leq d$, the function $R^t_{n,i}:U_n\to\bb R$ is defined by
\begin{equation*}
R^t_{n,i}(x)\,:=\,\frac{1}{2}\big\{\eta_t(x)+\eta_t(x+e_i)-2\eta_t(x)\eta_t(x+e_i)\big\}\,-\,\rho(t,m_n(x))\big[1-\rho(t,m_n(x+e_i))\big],
\end{equation*}
if $x\notin \Gamma_n^{i,-}$, and 
$R^t_{n,i}(x)\,:=\,0$
otherwise.

On the other hand, it is easy to check that
\begin{equation}\label{CD}
C_n+D_n\,=\,cn^{2-\theta}\sum_{x\in\Gamma_n}\frac{g([\frac{x}{n}])-\rho(t,m_n(x))}{\rho(t,m_n(x))\big[1-\rho(t,m_n(x))\big]}\big[\eta(x)\,-\,\rho(t,m_n(x))\big].
\end{equation}

\subsection{Computation of $\frac{\partial_t\Psi_t^N}{\Psi_t^n}(\eta)$}

A straightforward computation shows that
$$\frac{\partial_t\Psi_t^n}{\Psi_t^n}(\eta)=\sum_{x\in U_n}\sum_{i=1}^d\frac{\partial_t\rho(t,m_n(x))}{\rho(t,m_n(x))\big[1-\rho(t,m_n(x))\big]}\big[\eta(x)-\rho(t,m_n(x))\big]$$
Since $\rho(t,\cdot)$ is the solution of the heat equation and $m_n(x)\in U$ for every $x\in U_n$,
\begin{equation*}
\frac{\partial_t\Psi_t^n}{\Psi_t^n}(\eta)=\sum_{x\in U_n}\sum_{i=1}^d\frac{\partial_{u_iu_i}^2\rho(t,m_n(x))}{\rho(t,m_n(x))\big[1-\rho(t,m_n(x))\big]}\big[\eta(x)-\rho(t,m_n(x))\big].
\end{equation*}
By \eqref{rel1} \eqref{rel2},  for every $u\in U$,
\begin{equation}\label{rel3}
\frac{\partial_{u_iu_i}^2\rho(t,u)}{\rho(t,u)(1-\rho(t,u))}\,=\,\partial^2_{u_iu_i}H(t,u)\,+\,(\partial_{u_i}H(t,u))^2\big(1-2\rho(t,u)\big).
\end{equation}
Therefore we can conclude that
\begin{equation}\label{dtPsi}
\begin{split}
\frac{\partial_t\Psi_t^n}{\Psi_t^n}(\eta)\,=\,&\sum_{x\in U_n}\sum_{i=1}^d \partial^2_{u_iu_i}H(t,m_n(x))\big[\eta(x)-\rho(t,m_n(x))\big]\\
+\,&\sum_{x\in U_n}\sum_{i=1}^d(\partial_{u_i}H(t,m_n(x)))^2\big(1-2\rho(t,m_n(x))\big)\big[\eta(x)-\rho(t,m_n(x))\big]\\
\end{split}
\end{equation}

\subsection{Proof of Lemma \ref{order}}
Given $\ell \in \bb N$ and $x\in U_n$, let $B_{\ell}(x)$ be the set of collection of points in $U_n$ with distance to $x$ at most $\ell$:
$$B_{\ell}(x):=\big\{y\in U_n:\, |y-x|\leq \ell \,\big\}.$$ 
Let $\eta^{\ell}(x)$ be the average value of ${\eta}(\cdot)$ in $B_{\ell}(x)$:
$${\eta}^{\ell}(x):=\frac{1}{|B_{\ell}(x)|}\sum_{y\in B_{\ell}(x)}{\eta}(y).$$
Note that for the $x\in U_n$ which is close enough to the boundary $\Gamma_n$, the size of $B_{\ell}(x)$ may not equal $(2\ell+1)^d$. However, we always have 
$$(\ell+1)^d\,\leq\,|B_{\ell}(x)|\,\leq\,(2\ell+1)^d$$
for every $x\in U_n$. In addition, for any $\ell>0$ fixed, almost all the sites $x\in U_n$ satisfying that $|B_{\ell}(x)|\,=\,(2\ell+1)^d$: 
\begin{equation}\label{Vn}
\lim_{n\to\infty}\frac{|V_n|}{|U_n|}\,=\,1,
\end{equation}
where $V_n:=\{x\in U_n:\, |B_{\ell}(x)|\,=\,(2\ell+1)^d\}.$

Now we give the proof of Lemma \ref{order}.
\begin{proof}
In view of \eqref{4sum}\eqref{rel1}\eqref{AB}\eqref{CD}\eqref{dtPsi}, the common terms in $\frac{n^2L^*_n\Psi_t^n}{\Psi_t^n}$ and $\frac{\partial_t \Psi_t^n}{\Psi_t^n}$ are the ones with term $\partial_{u_i u_i}^2H(t,m_n(x))$, so  they cancel with each other. Therefore $\Big(\frac{n^2L^*_n\Psi_t^n}{\Psi_t^n}\,-\,\frac{\partial_t \Psi_t^n}{\Psi_t^n}\Big)(\eta)$ is equal to
\begin{equation}\label{fsum}
\begin{split}
&\,\sum_{i=1}^d\sum_{x\in\Gamma_n^{i,+}}\frac{n\partial_{u_i}\rho(t,m_n(x))+cn^{2-\theta}[g([x/n])-\rho(t,m_n(x))]}{\rho(t,m_n(x))\big(1-\rho(t,m_n(x))\big)}[\eta(x)-\rho(t,m_n(x))]\\
+&\,\sum_{i=1}^d\sum_{x\in\Gamma_n^{i,-}}\frac{-n\partial_{u_i}\rho(t,m_n(x))+cn^{2-\theta}[g([x/n])-\rho(t,m_n(x))]}{\rho(t,m_n(x))\big(1-\rho(t,m_n(x))\big)}[\eta(x)-\rho(t,m_n(x))]\\
+&\,\sum_{x\in U_n}\sum_{i=1}^d\big(\partial_{u_i}H(m_n(x))\big)^2\Big[R_{n,i}^t(x)\,-\,\big(1-2\rho(t,m_n(x))\big)\big(\eta(x)-\rho(t,m_n(x))\big)\Big]
\end{split}
\end{equation}
with an error term of order $o(n^d)$. 

We first introduce a set which can be ignored in the estimate of the above sum:
$$\Gamma_{n,\Delta}\,=\,\{x\in\Gamma_n:\exists\,\, \text{more than one coodinate}\,\,i, \,\text{s.t.}\,\,  (x_i-1)(x_i-n+1)=0 \}.$$
It is not hard to see that $|\Gamma_{n,\Delta}|=O(n^{d-2})$. Thus the total contribution of terms associated to points in $\Gamma_{n,\Delta}$ in the sum \eqref{fsum} is of order $o(n^d)$ for every $\theta\geq 0$. So we just need to consider the points not contained in $\Gamma_{n,\Delta}$.

By definition of $[\cdot]$, if $x\in\Gamma_n^{i,+}\setminus \Gamma_{n,\Delta}$, then $[x/n]$ is equal to $m_n(x-e_i)$ which belongs to the boundary $\Gamma$.  
For the case $\theta\in [0,1)$, the boundary condition of the corresponding hydrodynamic limit equation tells that $g([x/n])\,=\, \rho(t,m_n(x-e_i))$. In section \ref{sp}, we will show that $\rho(t,\cdot)\in C^\infty(\overline{U})$. Thus by Taylar expansion,
$$\rho(t,m_n(x))\,-\, \rho(t,m_n(x-e_i))\,=\, c^{-1}n^{\theta-1}\partial_{u_i}\rho(t,m_n(x))\,+\,o(1),$$
because $m_n$ is chosen to satisty $m_n(x)\,-\,m_n(x-e_i)\,=\,c^{-1} \,n^{\theta-1}$. This implies that the first term in the sum \eqref{fsum} is of order $o(n^d)$.  For the case $\theta=1$, the boundary condition of the corresponding hydrodynamic limit equation gives that, for every $x\in\Gamma_n^{i,+}$,
$$\partial_{u_i}^+\rho(t,m_n(x-e_i))\,=\,c\,\big[\rho(t,m_n(x-e_i))\,-\,g([x/n])\big]$$
Since $\rho(t,\cdot)\in C^\infty(\overline{U})$ and $m_n(x)-m_n(x-e_i)=n^{-1}$, 
$$\partial_{u_i}\rho(t,m_n(x))\,=\, \partial_{u_i}^+\rho(t,m_n(x-e_i))\,+\, o(1)$$
and
$$\rho(t,m_n(x))\,=\, \rho(t,m_n(x-e_i))\,+\, o(1).$$
These three identities above imply that the first term in the sum \eqref{fsum} is of order $o(n^d)$. Finally let us consider the case $\theta>1$. From the boundary condition of the corresponding hydrodynamic limit equation, for every $x\in\Gamma_n^{i,+}$,
$$\partial_{u_i}^+\rho(t,m_n(x-e_i))\,=\,0.$$
Combining the fact that $\rho(t,\cdot)\in C^\infty(\overline{U})$ and $m_n(x)-m_n(x-e_i)=n^{-1}$, 
$$\partial_{u_i}\rho(t,m_n(x))\,=\, o(1).$$
The term 
$$cn^{2-\theta}[g([x/n])-\rho(t,m_n(x))]$$ 
is of order $o(n)$ since $\theta>1$. Therefore we still have that the first term in the sum \eqref{fsum} is of order $o(n^d)$. The second term can be estimated in the same way so we omit the proof.

 On the other hand, since $H$ and $\rho$ are smooth functions, a summation by parts permits to replace $\eta(x)$ by $\eta^{\ell}(x)$ in the third term of \eqref{fsum}, up to an error term $o_{n,\ell}(1)\times n^d$, where $o_{n,\ell}(1)$ is a term which  vanishes after letting $n\to\infty$ then $\ell\to\infty$. Applying Lemma \ref{LDP}, we finish the proof.
\end{proof}

\begin{lemma}\label{LDP}
Let $\psi(x)=x(1-x)$ for $x\in\bb R$ and 
$$M(x,y)=\psi(x)\,-\,\psi(y)\,-\,\psi'(y)(x-y).$$
Then for every $0\leq t\leq T$,
\begin{equation}\label{oneblock}
\begin{split}
\frac{1}{n^d}&\bb E\Big[\Big\{\sum_{x\in U_n}\big(\partial_{u_i}H(m_n(x))\big)^2M(\eta_t^{\ell}(x),\rho(t,m_n(x))\Big\}\Big]\\
\leq\,&\frac{H(\mu_t^n|\nu^n_t)}{n^d}\,+\,o_{n,\ell}(1).
\end{split}
\end{equation}
\end{lemma}
\begin{proof}
First of all, by \eqref{Vn}, it is enough to prove the inequality \eqref{oneblock} with the sum of $x$ over $U_n$ replaced by $V_n$. By entropy inequality, the expectation in the lemma with sum over $V_n$ is bounded by 
\begin{equation*}
\frac{1}{\gamma }\left\{H(\mu_t^n|\nu_t^n)\,+\,\log E_{\nu_t^n}\Big[\exp\Big\{\gamma\sum_{x\in V_n}\big(\partial_{u_i}H(m_n(x))\big)^2M(\eta^{\ell}(x),\rho(t,m_n(x))\Big\}\Big]\right\},
\end{equation*}
for every $\gamma>0$. In the proof of this lemma we just need to choose $\gamma=1$. For each $x\in U_n$, let us define the function $F_{\ell}:U_n\to\bb R$
$$F_{\ell}(x)\,=\,\big(\partial_{u_i}H(m_n(x))\big)^2M(\eta^{\ell}(x),\rho(t,m_n(x)).$$
 To prove the lemma, it remains to show that, 
\begin{equation}\label{expldp}
\lim_{\ell\to\infty}\lim_{n\to\infty}\frac{1}{n^d}\log E_{\nu_t^n}\Big[\exp\Big\{\sum_{x\in V_n}F_{\ell}(x)\Big\}\Big]\,\leq\,0.
\end{equation}
Notice that for any $x,y\in U_n$ such that $|x-y|\geq2\ell +1$, $\eta^{\ell}(x)$ and $\eta^{\ell}(y)$ are independent under measure $\nu^n_t$. Therefore by Holder inequality,
\begin{equation}\label{holder}
\begin{split}
&\log E_{\nu_t^n}\Big[\exp\Big\{\sum_{x\in V_n}F_{\ell}(x)\Big\}\Big]\\
\,=\,&\log E_{\nu_t^n}\Big[\exp\Big\{\sum_{|z|\leq\ell}\sum_{\substack{x= z+(2\ell+1)y\in V_n\\\text{for some}\, z\in\bb Z^d}}F_{\ell}(x)\Big\}\Big]\\
\,\leq\,& \frac{1}{(2\ell+1)^d}\sum_{|z|\leq\ell}\log E_{\nu_t^n}\Big[\exp\Big\{(2\ell+1)^d\sum_{\substack{x= z+(2\ell+1)y\in V_n\\\text{for some}\,\, z\in\bb Z^d}}F_{\ell}(x)\Big\}\Big]\\
\,=\,& \frac{1}{(2\ell+1)^d}\sum_{x\in V_n}\log E_{\nu_t^n}\Big[\exp\Big\{(2\ell+1)^d F_{\ell}(x)\Big\}\Big].\\
\end{split}
\end{equation}
Consider a sequence of i.i.d. Bernoulli random variables $\{X_1, X_2,\cdots\}$. For each positive integer $n$, let $\overline{X}_n$ be the average of the first $n$ random variables:
$$\overline{X}_n\,=\,\frac{1}{n}\sum_{k=1}^n X_k.$$
Since $\rho(t,\cdot)$ and $H(t,\cdot)$ are smooth and $\nu_t^n$ is a product measure, there exists a constant $C=C(H)>0$ such that for all $x\in V_n$,
\begin{equation}\label{unifbound}
\begin{split}
\lim_{n\to\infty}&E_{\nu_t^n}\Big[\exp\Big\{(2\ell+1)^d F_{\ell}(x)\Big\}\Big]\\
\leq\,\sup_{\rho\in (\varepsilon_0,1-\varepsilon_0)}&E_{\nu_{\rho}}\Big[\exp\Big\{(2\ell+1)^d C(H)^2 M(\overline{X}_{\ell}, \rho)\Big\}\Big].
\end{split}
\end{equation}
where $\nu_\rho$ is the Bernoulli measure with mean $\rho$.
By Laplace-Varadhan Lemma, 
\begin{equation}\label{vl}
\begin{split}
&\lim_{\ell\to\infty}\frac{1}{(2\ell+1)^d}\log E_{\nu_{\rho}}\Big[\exp\Big\{(2\ell+1)^d C(H)^2 M(\overline{X}_{\ell}, \rho)\Big\}\Big]\\
&=\,\sup_{\lambda\in[0,1]}\Big\{C(H)M(\lambda,\rho)\,-\, I_{\rho}(\lambda)\Big\}
\end{split}
\end{equation}
where $I_\rho(\cdot)$ is the large deviations rate function given by 
$$I_\rho(\lambda)\,=\,\lambda\log\frac{\lambda}{\rho}\,+\,(1-\lambda)\log\frac{1-\lambda}{1-\rho}.$$
Note that $M(\cdot,\rho)$ and $I_\rho(\cdot)$ both vanish at $\lambda=\rho$. By computing the derivative, we see that $M(\cdot,\rho)$ attains its maximum at $\rho$ and $I_\rho(\cdot)$ attains its minimum at $\rho$. Thus the right hand side of \eqref{vl} is equal to $0$.  \eqref{expldp} follows from \eqref{holder} \eqref{unifbound} and \eqref{vl}.
\end{proof}

\section{Properties of the solution $\rho$}\label{sp}

The existence and uniqueness of the solution $\rho$ of equation \eqref{<1}(resp. \eqref{=1} and \eqref{>1}) is proved in Theorem 5.2(resp. Theorem 5.3) of Chapter IV in \cite{lsu68}. Moreover it is shown there that $\rho$ is smooth on $[0,T]\times\overline{U}$ under assumptions (1)(2)(3) given in subsection \ref{hle}.

In the following theorem, we establish the maximum principle. Then we apply it to give upper and lower bounds of the solution.
\begin{theorem}[Maximum Principle]
Assume that  $\rho(\cdot,\cdot)$ is the solution of PDE \eqref{<1} or \eqref{=1} or\eqref{>1} such that $\rho(\cdot,\cdot)$ is smooth on $[0,T]\times\overline{U}$. Then if  $\rho(\cdot,\cdot)$ is the solution of PDE  \eqref{<1} or \eqref{=1}, 
$$\max_{[0,T]\times \overline{U}}\rho(t,u)\,=\,\max\Big\{\max_{\overline{U}}\rho_0(u), \max_\Gamma g(u)\Big\};$$
if  $\rho(\cdot,\cdot)$ is the solution of PDE  \eqref{>1}, 
$$\max_{[0,T]\times \overline{U}}\rho(t,u)\,=\,\max_{\overline{U}}\rho_0(u).$$
\end{theorem}
\begin{proof}
First of all, by the weak maximum principle (see Theorem 8 of Chapter 7 in \cite{e} for instance), $\rho(t,u)$ attains its maximum at sets $\{0\}\times \overline{U}$ or $[0,T]\times \Gamma$. This already proves the conclusion for $\rho$ if it is the solution of PDE \eqref{<1}.

If  $\rho(\cdot,\cdot)$ is the solution of PDE  \eqref{=1}, we claim that if the maximum is attained at some point $(t^0,u^0)\in [0,T]\times\Gamma$, then $\rho(t^0,u^0)\leq g(u^0)$. Indeed, if $u^0\in \Gamma_i^+$ for some $0\leq i\leq d$, then we must have that $\partial_{u_i}^+ \rho(t^0,u^0)\leq0$ since $(t^0,u^0)$ is the maximum point, which implies that $\rho(t^0,u^0)\leq g(u^0)$ by the boundary condition.  If $u^0\in \Gamma_i^-$, then we  have that $\partial_{u_i}^+ \rho(t^0,u^0)\geq0$, which still implies that $\rho(t^0,u^0)\leq g(u^0)$. Therefore our claim holds.

Finally let us consider the case when $\rho(\cdot,\cdot)$ is the solution of PDE  \eqref{>1}. Define the function 
$$\rho_\epsilon(t,u)=\rho(t,u)\,-\,\epsilon\sum_{i=1}^d \big(u_i-\frac{1}{2}\big)^2$$
with $\epsilon>0$.
It is easy to check that $\partial_t \rho_\epsilon(t,u)\,\leq\, \Delta \rho_\epsilon(t,u)$ for all $u\in U$. Therefore by  Theorem 8 of Chapter 7 in \cite{e}, $\rho_\epsilon(t,u)$ attains its maximum at $\{0\}\times U$ or $[0,T]\times \Gamma$. Since for every point $u^0\in \Gamma^+$, a direct computation shows that $\partial_{u_i}^+ \rho_\epsilon(t,u^0)\,>\,0$, thus $\rho_\epsilon(t,u)$ cannot attain its maximum at the set $[0,T]\times \Gamma^+$. On the other hand, for every point $u^0\in \Gamma^-$, we have $\partial_{u_i}^+ \rho_\epsilon(t,u^0)\,<\,0$, thus $\rho_\epsilon(t,u)$ cannot attain its maximum at the set $[0,T]\times \Gamma^-$ neither. Therefore the maximum of $\rho_\epsilon(t,u)$ must be attained at $\{0\}\times U$. Letting $\epsilon\to 0$, noticing that $U$ is an open set, we conclude that $\rho(t,u)$ attains its maximum at $\{0\}\times \overline{U}$.
\end{proof}

Analogously we could prove the minimum principle and easily obtain the following corollary.
\begin{corollary}\label{col}
Under assumptions (1)(2)(3) given in subsection \ref{hle}, 
\begin{equation}\label{rhobound}
\rho(t,u)\in (\varepsilon_0,1-\varepsilon_0)
\end{equation}
for every $t\in[0,T]$ and $u\in\overline{U}$.
\end{corollary}
\smallskip\noindent{\bf Acknowledgments.} The author would like to thank the financial support of FAPESP Grant No.2019/02226-2.

\end{document}